\documentclass[12pt]{amsart}
\usepackage{geometry}                % See geometry.pdf to learn the layout options. There are lots.
\usepackage{graphicx}

\usepackage{amsmath}
\usepackage{amsfonts}
\usepackage{enumerate}
\usepackage{amssymb}
\usepackage{epstopdf}

\newtheorem{theorem}{Theorem}[section]
\newtheorem{proposition}[theorem]{Proposition}
\newtheorem{lemma}[theorem]{Lemma}

\theoremstyle{definition}
\newtheorem{definition}[theorem]{Definition}

\numberwithin{equation}{section}

\renewcommand{\epsilon}{\varepsilon}

\renewcommand{\rho}{\varrho}

\DeclareMathOperator{\HC}{HC} 
\begin{document}
\centerline{\bf Linear bounds for constants in Gromov's systolic inequality}
\centerline{\bf and related results}

%\centerline{\bf Appendix}

%\medskip\noindent
\vskip 0.5truecm
\centerline{\bf Alexander Nabutovsky}
%\institute{Department of Mathematics, University of Toronto}
%\maketitle
%
\vskip 0.5truecm

{\bf Abstract.} Gromov's systolic inequality asserts that the length, $sys_1(M^n)$, of the shortest non-contractible curve in a
closed essential Riemannian manifold $M^n$ does not exceed $c(n)vol^{1\over n} (M^n)$ for some constant $c(n)$. (Essential manifolds is a class of non-simply connected
manifolds that includes all non-simply connected closed surfaces, tori, and projective spaces.)

Here we prove that all closed essential Riemannian manifolds satisfy $sys_1(M^n)\leq n\ vol^{1\over n}(M^n)$. (The best previously known upper bound for $c(n)$ was exponential in $n$.) 

We similarly improve a number of related inequalities. The paper also contains a qualitative
strengthening of Guth's theorem from [Gu11], [Gu17] asserting that if volumes of all metric balls of radius $r$ in a closed Riemannian manifold
$M^n$  do not exceed $({r\over c(n)})^n$, then the $(n-1)$-dimensional Urysohn width of the manifold does not exceed $r$. In our version the assumption of Guth's theorem
is relaxed to the assumption that for each $x\in M^n$ there exists $\rho(x)\in (0,r]$ such that the volume of the metric ball $B(x,\rho(x))$
does not exceed $({\rho(x)\over c(n)})^n$, where one can take $c(n)={n\over 2}$.

\vskip 1.5truecm

\section{Introduction.}

Let $M^n$ be a closed Riemannian manifold. Larry Guth ([Gu 17]) proved that there exists $c(n)$
with the following property: if for some $r>0$ the volume of each metric ball of radius $r$ is less than 
$({r\over c(n)})^n$, then there exists a continuous map from $M^n$ to a $(n-1)$-dimensional simplicial complex such that the inverse image of each point can be covered by a metric ball of radius $r$ in $M^n$. It was previously proven by Gromov that this result implies
two by now famous Gromov's inequalities: $Fill Rad(M^n)\leq c(n)vol(M^n)^{1\over n}$ (Theorem 1.2.A in [Gr]) and, if $M^n$ is essential, then also $sys_1(M^n)\leq 6c(n)vol(M^n)^{1\over n}$ (Theorem 0.1.A in [Gr]) with the same constant $c(n)$.

Here $sys_1(M^n)$ denotes the length of a shortest
non-contractible closed curve in $M^n$. 

Here we prove that these results hold with $c(n)=({n!\over 2})^{1\over n}\leq {n\over 2}$. We demonstrate that
for essential Riemannian manifolds $sys_1(M^n) \leq n\ vol^{1\over n}(M^n)$. %={n\over e}(1+o(1))$. 
All previously known upper bounds for $c(n)$ were exponential in $n$.

Moreover, we present a qualitative improvement: In Guth's theorem the assumption that the volume
of every metric ball of radius $r$ is less than $({r\over c(n)})^n$ can be replaced by a weaker assumption that for every point
$x\in M^n$ there exists a positive $\rho(x)\leq r$ such that the volume of the metric ball of radius $\rho(x)$ centered at $x$ is less than $({\rho(x)\over c(n)})^n$
(for $c(n)=({n!\over 2})^{1\over n}$). 

Also, if $X$ is a boundedly compact metric space such that for some $r>0$ and an integer $n\geq 1$  
the $n$-dimensional Hausdorff content of each metric ball of radius $r$ in $X$
is less than $({r\over 4n})^n$, then there exists a continuous map from $X$ to a $(n-1)$-dimensional simplicial complex such that the inverse image of each point can be covered by a metric ball of radius $r$. This provides a (significant) quantitative improvement of a result
from [LLNR] and [P]. (Recall that a metric space is called {\it boundedly compact} if all closed metric balls in this space are compact.)

Most other papers in systolic geometry follow Gromov's approach based on the isoperimetric inequality in Banach spaces proven using
``cutting off of thin fingers". We follow Schoen-Yau style approach , i.e. the inductive dimension reduction. This approach was introduced to systolic geometry by Guth ([Gu10]) and later greatly improved and strengthened by Papasoglu ([P]) who used some ideas from [LLNR].
Our approach in the present paper essentially follows [P], yet we provide a number of modifications, strengthenings and simplifications of proofs there.

Our paper is almost self-contained. Without a proof we use only two (well-known) facts that go beyond material taught in standard graduate courses, namely,  the existence of smooth approximations of the distance function (cf. [Ga]) and the coarea formula (cf. [BZ], Theorem 3.2.4).

%Our very short and almost self-contained proofs  use ideas from recent paper [P] by Panos Papasoglu.

%\vskip 1.5truecm
%\par\noindent
%{\bf 1. Introduction.} 

\section{Results.}
%
%Recall that the $m$-dimensional Hausdorff content of a compact subset $C$ of a metric space $X$ is defined as the infimum over all coverings of $C$ by closed metric balls in $X$ of $\sum_i r^i$, where $r_i$ are the radii of all balls of the covering. The Urysohn $k$-width is, by definition, the infimum
%of $\sup_{x\in K^k}diam(f^{-1}(x))$, where the infimum is taken over the set of all continuous maps $f$ of $X$ to all
%$k$-dimensional simplicial complexes $K$. Intuitively, it is a quantitative measure of how close $X$ is to
%being $k$-dimensional. 
\subsection{Definitions and historical context.}
Given a bounded metric space $X$ its Kuratowski embedding into $L^{\infty}(X)$ sends each point $x$ to the distance function to $x$.
%If $f_*(H_k(X^n;G))\not= \{0\}$ for some $k\geq m$,
%then $X^n$ is called $m$-essential. 
Gromov defined the {\it filling radius} of a closed Riemannian manifold $M^n$, $Fill Rad (M^n)$ as the infimum of $r$ such that the image of $M^n$ in $L^{\infty}(M^n)$ under the Kuratowski embedding bounds
in its $r$-neighbourhood ([Gr], section 1). In [Gr] Gromov gave a proof of the inequality $Fill Rad(M^n)\leq c(n)vol(M^n)^{1\over n}$
with the constant that behaves as $(Cn)^{3n\over 2}$. (On the other hand Misha Katz's paper [K] contains a short proof the inequality $FillRad(M^n)\leq {1\over 3}diam(M^n)$ with the optimal constant.)
Gromov's proof was later somewhat simplified by Stefan Wenger ([W]). (More precisely, Wenger simplified the proof of Gromov's filling volume inequality which is the main ingredient of the proof of inequality $Fill Rad(M^n)\leq c(n)vol(M^n)^{1\over n}$ in [Gr].)

An $n$-dimensional simplicial complex $X^n$ is {\it essential}, if there is no map $f:X^n\longrightarrow K(\pi_1(X^n),1)$ such that $f$ induces the isomorphism of the fundamental groups,
and the image of $f$ is contained in the $(n-1)$-skeleton of $K(\pi_1(X^n),1)$. It is easy to see that $X^n$ is essential if the classifying map
$X^n\longrightarrow K(\pi_1(X^n),1)$ 
induces the homomorphism of $n$th homology groups with non-trivial image (for some group of coefficients). 

The paper [Gr] contains a short
and elegant proof of the inequality $sys_1(M^n)\leq 6 Fill Rad(M^n)$ for all essential closed Riemannian
manifolds (Lemma 1.2.B in [Gr]). Combining this inequality with $Fill Rad(M^n)\leq c(n)vol^{1\over n}(M^n)$ Gromov proves that
all closed essential Riemannian manifolds satisfy $sys_1(M^n)\leq c(n)vol^{1\over n}(M^n)$.
This inequality generalizes earlier results by Loewner, Pu, Accola, Blatter, Hebda, and Yu. Burago and V. Zalgaller 
for surfaces. In partucular, Yu. Burago and V. Zalgaller and, independently, Hebda proved that for all closed Riemannian surfaces $\Sigma$, $sys_1(\Sigma)\leq\sqrt{2}\sqrt{\rm{Area}(\Sigma)}$ (cf. [BZ]).

One can define the $(n-1)$-dimensional {\it Urysohn width}, $UW_{n-1}(X)$,  of a metric space $X$ as infimum of $t$ such that there exists a continuous map
$f:X \longrightarrow K^{n-1}$ to a $(n-1)$-dimensional polyhedron $K^{n-1}$ such that for each $k\in K^{n-1}$
$f^{-1}(k)$ has diameter $\leq t$. We will need also a closely related notion of $(n-1)$-dimensional {\it Alexandrov width}, $UR_{n-1}(X)$, that is
defined almost as the Urysohn width, but with condition $diam(f^{-1}(k))\leq t$ replaced by
the condition that $f^{-1}(x)$ is contained in a metric ball of radius $t$. It is obvious that $UR_{n-1}\leq UW_{n-1}(X)\leq 2UR_{n-1}(X)$.
The $(n-1)$-dimensional Alexandrov width of a compact metric space $X$ can also be defined as the infimum of $t$ such that there exists a covering of $X$ by connected open sets $U_\alpha$ of radius
$\leq t$ such that no $n+1$ sets $U_\alpha$ have a non-empty intersection. (The equivalence of these two definitions is well-known.
For the sake of completeness, we include a sketch of a proof of the equivalence in section 1.3 below.)
%To see that these two definitions are equivalent for all compact $X$, denote $UR_{n-1}(X)$ in the sense of the first definition by $\rho$, and the second by $r$.
%We will first demonstrate that $r\leq\rho+\epsilon$ for
%an arbitrarily small positive $\epsilon$, and then demonstrate that $\rho\leq r$. To prove the first
%inequality choose $K^{n-1}$ and $f$ such that $\max_{k\in K^{n-1}}diam(f^{-1}(k))$ is very close to $\rho$. 
%Now consider a very fine covering of $K^{n-1}$ 
%by open sets $V_\beta$ such that each $(n+1)$-tuple of these sets has the empty intersection.
%Finally, define the collection $U_\alpha$ as the collection of all connected components of open sets $f^{-1}(V_\beta)$. Note that 
%the diameters of $U_\alpha$ do not exceed $\max_{k\in K^{n-1}}diam(f^{-1}(k))+\epsilon$, where $\epsilon>0$ can be made arbitrarily small
%by choosing the covering $V_\beta$ sufficiently fine.
%To see that $\rho\leq r$, consider the nerve $K^{n-1}$ of the covering $U_\alpha$,
%and the standard map $f:X\longrightarrow K^{n-1}$ defined using a partition of unity subordinate to the covering
%$\{U_\alpha\}$. Now note that the inverse images of each point of $K^{n-1}$ under $f$ 
%will be contained in one of the sets $U_\alpha$. 

Gromov  also
provided a proof of the inequalities $Fill Rad(M^n)\leq {1\over 2}UW_{n-1}(M^n)$ (the combination
of Proposition (D) in Appendix 1 in [Gr] with the inequality in the example at the end of section (B) in Appendix 1 of [Gr]).
Therefore, any upper bound for $UW_{n-1}$ automatically leads to upper bounds to 
$Fill Rad$ and, in the essential case, for $sys_1$. Now a natural question
(posed by Gromov in [Gr]) is whether or not
$UW_{n-1}(M^n)\leq c(n)vol(M^n)^{1\over n}$. This question was solved in
the affirmative by Larry Guth in [Gu11], [Gu17]. 
In fact, Guth proved more. He demonstrated that there exists $\delta(n)$ such that if for some $r>0$ all metric balls of radius $r$ have volume less than $\delta(n)r^n$, then $UW_{n-1}(M^n)\leq r$. To recover the previous inequality one can take here
$r={vol(M^n)^{1\over n}\over \delta(n)^{1\over n}}$. The assumption will automatically hold, and one sees that
$UW_{n-1}(M^n)\leq \delta(n)^{-{1\over n}}vol(M^n)^{1\over n}$.

Recall that the $m$-dimensional Hausdorff content of a compact metric space $X$ is the infimum over all coverings of $X$ by metric
balls with radii $r_i$ of the sum $\sum_i r_i^m$. It is denoted by $HC_m(X)$. (If one requires here that all $r_i$ do not exceed $\delta$, and then takes the limit as
$\delta\longrightarrow 0$, one obtains the $m$-dimensional Hausdorff measure of $X$.)
Guth asked if one can replace the volume in these inequalities by the $n$-dimensional Hausdorff content, and if such estimates
will be true for all (not necessarily $n$-dimensional) compact metric spaces (Questions 5.1 and 5.2 in [Gu 17]).

In [LLNR] we proved that this is, indeed, so. %Namely, we proved
%that for each positive integer $m$ and each boundedly compact metric space $X$ 
%there exists $\epsilon_m>0$ such that if for some $r$ every metric ball in $X$ has $m$-dimensional Hausdorff content $\leq\epsilon_mr^m$, then the $(m-1)$ dimensional Urysohn width of $X$ does not exceed $r$ (Theorem 1.1 in [LLNR]).
%If $X$ is compact, then choosing $r={HC_m(X)^{1\over m}\over\epsilon_m^{1\over m}}$, we immediately see that
%$UW_{m-1}(X)\leq c(m)\HC_m(X)^{1\over m}$, where $c(m)={1\over\epsilon_m^{1\over m}}$ (Theorem 1.2 in [LLNR]).
For example, we proved that for each compact metric space $X$ and each integer $m>1$, $UW_{m-1}(X)\leq C(m)HC_m^{1\over m}(Y)$.
As a corollary, we immediately see that
if $X$ is a compact $m$-essential smooth polyhedron 
endowed with the structure of the length space, $sys_1(X)\leq 3C(m)\HC^{1\over m}(X)$.
%
%Several months after the first draft of our paper [LLNR] was posted online, Panos Papasoglu devised a much simpler proof of our Theorem 1.1 ([P]).
%Moreover, his proof yields a better constant in Theorem 1.1 (and as a corollary, in Theorems 1.2 and 1.4). While he did not provide
%an estimate of the dimensional constants, his proof
%yields constants of the form $const^{-m^2/2}$ in Theorem 1.1, and  $const^m$ in Theorems 1.2 and 1.4.
%Our proof in [LLNR] yields, respectively,  $(const\ m)^{-m^2}$ and $(const\ m)^m$. Note, that the best previously known upper
%bounds for constants in Gromov's filling radius and systolic inequalities also
%behave as $(Cn)^n$ ([W]).
Recently, Panos Papasoglu wrote a paper [P] with a much shorter proof of these results than the proof in [LLNR]. His proof did not contain an estimate for $C(m)$,
but our analysis of his proof yields $C(m) \sim const^m$ leading to exponential in $m$ estimates for
constants in the previous inequalities. He learned about [LLNR] from my talk and our conversations at the conference at Barcelona.
%His
%proof draws on several ideas of [LLNR].
While his proof draws on
several ideas of [LLNR], it also contains
%In particular, he uses Eilenberg's inequality for Hausdorff contents discovered in [LLNR], a trick used
%in [LLNR] to counter the non-additivity of the Hausdorff content and an argument based on the Tietze extension to prove a weaker version of Lemma 2.5 below.
a central observation that is quite different from the ideas of [LLNR]. Roughly speaking, the amazing in its strength and simplicity Papasoglu's insight was to consider an (almost) minimal ``hypersurface" dividing a compact metric space into subsets of a small diameter and to observe that the ``area" (or, more precisely, the appropriate Hausdorff content) of the intersection of this hypersurface
with any metric ball cannot exceed the Hausdorff content of the corresponding metric sphere. Indeed, otherwise, one could just replace
the part of the minimal hypersurface inside the metric ball by the metric sphere preserving the same upper bound for diameter for each component of the complement. Thus, the (almost) minimal hypersurface inherits the main property of the metric space, namely, that its intersections with metric balls of a certain size are ``small".
This observation enables one to run an induction argument, where the result for the metric space and the $n$-dimensional
Hausdorff content would follow from the same result for the minimal hypersurface and its $(n-1)$-dimensional Hausdorff content.

Of course, this approach is strongly reminiscent of
of the Schoen - Yau approach to scalar curvature that was introduced to systolic geometry by Guth, who in [Gu10] proved
that $sys_1(M^n)\leq 8nvol(M^n)^{1\over n}$ for Riemannian tori.  
The proof of our main theorem below is heavily based on Papasoglu's idea. Yet it
contains a number of modifications and simplifications:

First, we observed that the dependence of the constant in the inequalities could be improved from exponential to linear by (a) carefully choosing the radius
of the metric ball (in the argument above) and (b) improving an argument of the end of the proof of Lemma 2.4 in [P] so as not to decrease the constant
by a constant factor on each step - compare our Lemma 2.5, where the upper bound in the assumption and the conclusion is the same. (We observed that it is more convenient to use $UR$ instead of the previously used $UW$ here.) In order to accomplish (a), we could have used the trick used by Larry Guth at the end of
[Gu10] (as it was done in the first version of the present paper). 

However, we noticed that there is a better (point dependent) way to choose the radii as in Lemma 2.4 below. Not only this observation leads to an improvement of the estimate by a constant factor, but it also yields a quantitative improvement of all the previous results that was mentioned in the abstract: the radius
of a small ball centered at a point is allowed to depend on the center as long as it does not exceed a fixed $r$.
It is interesting to note that we do not see how to achieve this quantitative
improvement, if one follows the approach of [P] via Hausdorff contents, as in this approach  one needs to restrict the radii of the considered balls
by a quantity that depends on the radii of small balls $r$ and becomes wildly variable, if these radii are allowed to depend on the centers. (See the remark
at the end of section 3 for more details.)

Third, the approach of Papasoglu to the classical systolic geometry was through results about Hausdorff contents (the same as in [LLNR]). He mentioned that instead one can directly use the Hausdorff measure and Eilenberg's inequality. We adopted this
approach and discovered that not only it leads to a much simpler proof, but also one can save an extra $\sqrt{n}$ factor in comparison with first establishing the inequality for Hausdorff content with linear constant (Theorem 1.4), and then using the obvious inequality
relating Hausdorff measure and Hausdorff content.
%as the coarea inequality for Riemannian polyhedra is stronger than Eilenberg's inequality. 

Fourth, we were careful about the values of the numerical 
constants in our proof. This is, probably, not that important in the long run, as one expects that the optimal
dimensional constants in the above inequalities should behave as $const\ \sqrt{n}$ 
and not as $const\ n$. Still, as the result, we derive
aesthetically pleasing and convenient to use inequalities $sys_1(M^n)\leq n\ vol^{1\over n}(M^n)$ for all closed
essential manifolds, and $Fill Rad(M^n)\leq {n\over 2}vol^{1\over n}(M^n)$ for all closed manifolds. In fact, I am not aware of any previously published specific value of the constant at $vol(M^n)^{1\over n}$ in the general case of Gromov's systolic inequality for $n=3$ other than Gromov's $1296\sqrt{6}$ or Wenger's $118098$. For $n=3$ our value $ 2*3^{1\over 3}=2.88\ldots $ of the systolic constant is within the factor of $2$ of the (unknown) optimal value.
%One important modification is stronly reminiscent of what was done in the remark at the end of [Gu10].

Fifth, in the last section of the paper we also similarly improve the main result of [LLNR] and [P]. %We replace
%the assumption that all metric balls of some radius $r$ have $m$-dimensional Hausdorff content $\HC_m$ considerably than the $m$-th power of ${r\over c(m)}$ by a weaker assumption that this happens for some radius $t\leq r$ that depends on the center
%of the ball, and simultaneously 
We provide a four-page long self-contained proof of the inequality $UW_{m-1}(X)\leq 8m HC_m^{1\over m}(X)$ for a compact metric space $X$
as well as a local version of this result for boundedly compact $X$. This section heavily relies on [P] but contains several improvements and simplifications.

\subsection{Results.} Our first main theorem is:

\begin{theorem}
Let $M^n$ be a closed Riemannian manifold, and $r>0$ a real number. Assume that for every $x\in M^n$ there exists $t=t(x)\in (0,r]$ such that
the volume of the metric ball of radius $t$ centered at $x$ is
less than ${2t^n\over n!}$.  Then $UR_{n-1}(M^n)< r.$
\end{theorem}

The relationships between $UR_{n-1}$, $UW_{n-1}$, $Fill Rad$ and $sys_1$ stated above immediately imply that:

\begin{theorem} Assume that $M^n$ is a compact $n$-dimensional smooth Riemannian manifold. Then %(or, more generally, a smooth Riemannian polyhedron). Then
\par\noindent
$$UR_{n-1}(M^n)\leq ({n!\over 2})^{1\over n}vol(M^n)^{1\over n}\leq {n\over 2}\ vol^{1\over n }(M^n),\ \ \ (1)$$
\par\noindent
$$UW_{n-1}(M^n)\leq 2({n!\over 2})^{1\over n}vol(M^n)^{1\over n}\leq n\ vol^{1\over n}(M^n),\ \ \ (2)$$
%\par\noindent
%and also $$UW_{n-1}(M^n)\leq 2\sqrt{e}n\ |M^n|^{1\over n};\ \ \  (3)$$
$$Fill Rad(M^n)\leq  ({n!\over 2})^{1\over n} vol(M^n)^{1\over n}\leq {n\over 2}\ vol^{1\over n}(M^n),\ \ \ (3)$$
\par\noindent
%and asymptotically
%$$Fill Rad(X^n)\leq n(1+o(1)) |X^n|^{1\over n};$$
%\par\noindent
%and, if $M^n$ is essential, then also
%$$sys_1(M^n)\leq 6(n!)^{1\over n} vol(M^n)^{1\over n}\leq {6\over e}(1+o(1))nvol(M^n)^{1\over n}\ \ \ (4).$$
%and
%$$sys_1(X^n)\leq 6\sqrt{e} n|X^n|^{1\over n}.$$
%$$Fill Rad(X^n)\leq \sqrt(e) n\ |X^n|^{1\over n},$$
%\par\noindent
%and asymptotically
%$$Fill Rad(X^n)\leq {n+1\over e}(1+o(1))\ vol(X^n)^{1\over n};$$
%\par\noindent
%as well as
%$$sys_1(X^n)\leq {6(n+1)\over e}(1+o(1))vol(X^n)^{1\over n},$$
%and
%$$sys_1(X^n)\leq 4.5(n+1)\ vol(X^n)^{1\over n}\leq 7nvol(X^n)^{1\over n}.$$
%\par\noindent
\end{theorem}
\par\noindent
%{\bf Remarks.} 
%{\bf 1.} Roman Karasev e-mailed to the author that 
%the inequality $sys_1\leq 2UR_{n-1}$ for essential manifolds (or Riemannian polyhedra) seems to follow from the work of Albert Schwarz [S]. Indeed, it is known that one can also
%define $UR_{n-1}(M^n)$ as the infimum of $r$ such such that there exists a cover of %$M^n$
%by open sets $U_\alpha$ with radii $\leq r$ with multiplicity of the covering $\leq n$. If $sys_1(M^n)$ does not exceed $2r$, the image of $\pi_1(U_\alpha)$ in $\pi_1(M^n)$ is trivial. Therefore, the {\it Schwarz genus} of the universal covering of $M^n$ (introduced
%in [S]) does not exceed $n$. According to [S] this implies that the classifying map $M^n\longrightarrow B\pi_1(M^n)$ factors
%through an $(n-1)$-dimensional complex, and, therefore, $M^n$ is not essential. Note that the constant in this inequality is three times better than the constant in the corresponding Gromov's inequality. As a corollary, one obtains the systolic inequality with a better constant:
%$$sys_1(M^n)\leq 2(n!)^{1\over n} vol(M^n)^{1\over n}\leq {2\over %e}(1+o(1))nvol(M^n)^{1\over n}.\ \ \ \ (+)$$
%However, Schwarz's paper [S] is long and not easy to read. Therefore, at the end of the paper I will provide a short and self-contained argument proving (+) without using results or ideas from [S].\bf 2.} At the end of the paper we are going to present a completely elementary
Inequality (1) and the inequality $sys_1(X^n)\leq 6UR_{n-1}X^n)$ for essential 
Riemannian polyhedra which is a combination
of two inequalities: $sys_1(X^n)\leq 6FillRad(X^n)$ and $FillRad(X^n)\leq UR_{n-1}(X^n)$ that were proven in [Gr] immediately imply the following version
of Gromov's systolic inequality with linear in $n$ dimensional constant: For each essential Riemannian manifold $M^n$ $sys_1(M^n)\leq 3nvol^{1\over n}(M^n)$
that appeared in the first version of the present paper.
However, Roman Karasev e-mailed to me a very short proof of a stronger inequality:
$$sys_1(X^n)\leq 2UR_{n-1}(X^n).\ \ \ \ \ \ \eqno(4)$$
for essential polyhedral length spaces that does not involve the
Kuratowski embedding or the filling radius. Karasev's proof is based on work of Albert Schwartz [S], and in a nutshell goes as follows: 
%Theorem 14 in [S]
%implies that if $M^n$ is essential, then there is no open cover of $M^n$ by connected
%open sets $U_\alpha$ such that: 1) multiplicity of intersections of %$U_\alpha$
%does not exceed $n$; 2) If $p:\tilde M^n\longrightarrow M^n$ is the universal
%covering of $M^n$, then the restriction of $p$ to each connected component
%of $p^{-1}(U_\alpha)$ is a homeomorphism. On the other hand 
The second definition
of Alexandrov width implies that there is an open covering
of $M^n$ by connected open sets of radius $\leq UR_{n-1}(M^n)$ with multiplicity of intersections $\leq n$. If $sys_1(M^n)>2UR_{n-1}(M^n)$, then each loop in $U_\alpha$
is contractible in $M^n$. Therefore, 
each $U_\alpha$ lifts to $\tilde M^n$ as the collection of  sets
$\{(\tilde U_\alpha)_g\}_{g\in\pi_1(M^n)}$ homeomorphic to $U_\alpha$.
The collection of all these sets forms a covering of $M^n$ of multiplicity $\leq n$.
Theorem 14 in [S] asserts that the existence of such a covering of $\tilde M^n$
implies that $M^n$ is not essential.
%each open set $U_\alpha$
%can be lifted to a (disconnected) open set $\tilde U_\alpha$
%in $\tilde M^n$, and the collection of sets $\tilde U_\alpha$
%form an open covering of $\tilde M^n$ the existence of which contradicts the conclusion of Theorem 14 of [S].
With Karasev's permission I will present a self-contained proof of inequality (4) at the end of section 2.

As an immediate corollary:

\begin{theorem}
If $M^n$ is a closed essential Riemannian manifold, then $sys_1(M^n)\leq  2({n!\over 2})^{1\over n}vol^{1\over n}(M^n)\ \leq  n\ vol^{1\over n}(M^n)$. 
\end{theorem}
%essential Riemannian manifolds satisfy
%$$sys_1(M^n)\leq 2(n!)^{1\over n}vol(M^n)^{1\over n}={2\over e}(1+o(1)) n vol(M^n)^{1\over n}.\ \ \ (4')$$
\par\noindent
%{\bf 2.} 
%Note that for each $n\geq 2$, $(n!)^{1\over n}\leq {n\over \sqrt{2}}$. Therefore, 
%$Fill Rad(M^n)\leq {1\over\sqrt{2}}nvol(M^n)^{1\over n}$. If $M^n$ is essential, then the inequality ($4'$) implies that
%$sys_1(M^n)\leq \sqrt{2}nvol(M^n)^{1\over n}$.
\par\noindent
{\bf Remark.} As $(n!)^{1\over n}={1\over e}(1+o(1))n,$ for all sufficiently large $n$ $sys_1(M^n)< 0.74\ n\ vol(M^n)^{1\over n}$. If $n=2$, the inequality in the theorem is well-known, and a better estimate can be found in section 1.4.3 of [BZ]. If $n=3$, then
%\begin{corollary}
%If $M^3$ is a a closed essential Riemannian manifold, then $sys_1(M^n)\leq 6^{1\over3}$
%\end{corollary}
%Note that 
the constant at $vol^{1\over 3}(M^3)$ in Theorem 1.3 is equal to $2*3^{1\over 3}=2.88\ldots$. On the other hand, we see that the optimal value of this constant for $n=3$ cannot be less than $\pi^{1\over 3}=1.46\ldots$, as this is the value that one gets in the case of $RP^3$ with the canonical metric. So, for $n=3$, our constant is within the factor of $1.97$ from the optimal systolic constant. %If $n\geq 4$, $sys_1(M^n)< 0.56\ n\ vol^{1\over n}(M^n)$, and for $n\geq 6$ $sys_1(M^n) < {n\over 2}\ vol^{1\over n}(M^n)$.
%{\bf 2.} At the end of the paper we are going to present a completely elementary
%proof of Gromov's inequality $sys_1(X^n)\leq 6UR_{n-1}X^n)$ 
%for essential Riemannian polyhedra that does not involve the
%Kuratowski embedding or the filling radius.
\vskip 0.3truecm

We will prove Theorems 1.1, 1.2 in a somewhat greater generality, namely for compact Riemannian polyhedra (i.e. finite polyhedra endowed
with a smooth Riemannian metric on each maximal simplex, so that Riemannian metrics on two simplices that have a common face match on this face). Note that all previous definitions and quoted results by Gromov can be directly extended to Riemannian polyhedra, 
which was observed by Gromov in [Gr]. 

Below a subpolyhedron will always mean a compact subpolyhedron with smoothly embedded faces
endowed with the Riemannian metric of the ambient Riemannian polyhedron (and the corresponding intrinsic distance).
Also, below $|X|$ will denote the volume of $X$. % i.e. $n$-dimensional Hausdorff measure for $n=dim (X)$.
Sometimes we write it as $|X|_n$, when we want to emphasize the dimension.

In the last section, we give a self-contained proof of the following quantitative improvement of a result that first appeared in [LLNR] and then was reproven in [P]:

\begin{theorem}
\par\noindent
1. Let $X$ be a compact metric space, $r>0$, $n$ a positive integer. Assume that for each metric ball $B$ of radius $r$ in $X$, $HC_n(B)<({r\over 4n})^n$.
Then $UR_{n-1}(X)<r$.
\par\noindent
2. Let $X$ be a compact metric space. Then $UR_{n-1}(X)\leq 4n\HC_n(X)^{1\over n}$.
\par\noindent
3. Let $X$ be boundedly compact. Assume that for some positive $\mu$ and each metric ball $B$ of radius $r$, $HC_n(B)\leq({r\over 8n})^n-\mu$.
Then $UR_{n-1}(X)<r$.
\end{theorem}

\par\noindent
{\bf Remark.} As $UW_{n-1}(X)\leq 2UR_{n-1}(X)$, we also immediately obtain the corresponding upper bounds for the Urysohn width of $X$, $UW_{n-1}(X)$, that differ from the upper bounds for $UR_{n-1}(X)$ by a factor of $2$. For example, if $X$ is compact, then $UW_{n-1}(X)\leq 8nHC_n^{1\over n}(X)$.

\subsection{ Equivalence of the two definitions of the $(n-1)$-dimensional Alexandrov widths.} 

To see that the two definitions of Alexandrov width given in section 1.1 are equivalent for all compact $X$, denote $UR_{n-1}(X)$ in the sense of the first definition by $\rho$, and the second by $r$.
We will first demonstrate that $r\leq\rho+\epsilon$ for
an arbitrarily small positive $\epsilon$, and then demonstrate that $\rho\leq r$. To prove the first
inequality choose $K^{n-1}$ and $f$ such that $\max_{k\in K^{n-1}}rad(f^{-1}(k))$ is very close to $\rho$. 
Now consider a very fine covering of $K^{n-1}$ 
by open sets $V_\beta$ such that each $(n+1)$-tuple of these sets has the empty intersection.
Finally, define the collection $U_\alpha$ as the collection of all connected components of open sets $f^{-1}(V_\beta)$. Note that 
the radii of $U_\alpha$ do not exceed $\max_{k\in K^{n-1}}rad(f^{-1}(k))+\epsilon$, where $\epsilon>0$ can be made arbitrarily small
by choosing the covering $V_\beta$ sufficiently fine.
To see that $\rho\leq r$, consider the nerve $K^{n-1}$ of the covering $U_\alpha$,
and the standard map $f:X\longrightarrow K^{n-1}$ defined using a partition of unity subordinate to the covering
$\{U_\alpha\}$. Now note that the inverse images of each point of $K^{n-1}$ under $f$ 
will be contained in one of the sets $U_\alpha$.

\section{ Proof of Theorems 1.1 and 1.3.}

%All these inequalities are not sensitive
%to $C^1$-small perturbations of the Riemannian metric. As we can arbitrarily closely approximate a smooth Riemannian metric
%by an analytic one, we can assume that the Riemannian metric $M^n$ is analytic.
%In this case the distance function to a point
%is well-known to be subanalytic and, therefore,
%its level sets (=metric spheres) are subanalytic polyhedra. This will imply that all sets that we will be dealing with
%below are subanalytic polyhedra. (Another way to ensure that the level sets are polyhedra is to (piece-wise) smooth out the distance function and to use Sard's theorem, as it was done for similar purposes in [Gr].) 
%Our coarea inequality (Lemma 5.3 of [LLNR]) no longer applies,
%yet is is well-known that its version for volumes holds even without the constant $2$ in the right-hand side:
The well-known coarea inequality for Lipschitz functions of Riemannian manifolds immediately generalizes
to Riemannian polyhedra ([BZ]) and implies that
given a Riemannian polyhedron $X^n$, a real $r$, and a metric ball $B$ of radius $r$ centered at a point $x\in X^n$,
$\int_0^r |S_s|_{n-1}ds\leq |B|_n$, where $S_s$ denotes the metric sphere of radius $s$ centered at $x$. 

We prefer to work
in the situation when for almost all $s$, $S_s$ is a subpolyhedron. One well-known way to achieve this is to approximate the distance function by $(1+\tau)$-Lipschitz function that is smooth
on each open simplex (cf. section 3 of [Ga]), and to replace the distance function by this approximation. (Here $\tau$ can be chosen to be arbitrarily small.) Starting from Lemma 2.2 ``metric spheres"
will really mean the level sets of a sufficiently close smooth approximation of the distance function. Now Sard's theorem implies
that almost all geodesic spheres are subpolyhedra. (Sard's theorem will separately apply to the restriction of the smooth
approximation of the distance function to each open simplex.)
However, the coarea inequality above
and all inequalities below
will hold only up to a factor of $1+f(\tau)$, %where $\tau>0$ can be chosen to be arbitrarily small,
where $f$ will be some specific function such that $\lim_{\tau\longrightarrow 0}f(\tau)=0$. Eventually, one will pass
to the limit as $\tau\longrightarrow 0$. 
For the sake of readability we will not be mentioning terms of the form $1+f(\tau)$ in the inequalities, and will just pretend that the distance function is smooth on each open simplex.

\begin{lemma}
Let $X$ be a compact Riemannian polyhedron of dimension $\leq 1$. Assume that there exists $r>0$ such that for
each $x\in X$ there exists a metric ball $B$ centered at $x$ of radius $t(x)\in (0,r]$ such that $|B|<2t(x)$.
%Assume that for every metric ball $B$ of radius $r$ in  $X$ $|B|< r$.
Then $UR_0(X)< r$. In other words each connected component of $X$ can be covered by a 
a metric ball of radius $<r$.
\end{lemma}

\begin{proof}
%The coarea inequality implies that for each point $x\in X$ there exists $\rho<t(x)$ such 
%the geodesic sphere of radius $\rho$  centered at $x$ is empty. 
%This means that $X$ is the union of a family of disjoint closed metric balls with radii $<r$. 
%%%%%%%%
%We would like to prove that each connected component of $X$ can be covered by a closed metric ball of radius $<{r\over 2}$.
First, note that without any loss of generality we can assume that $X$ is connected.

%Without any loss of generality we can assume that $X$ is connected,
%and, therefore, is contained in just one metric ball of radius
%$r$ centered at $x$. So, we have a weaker assertion for balls
%of radius $<r$ instead of $<{r\over 2}$.
%%%%%%%%
%Use the compactness of $X^1$ to find a finite subcover $B_1,\ldots , %B_N$ with radii $r_1,\ldots ,r_N$.
%Now for each $i$ one can map $B_i\bigcap X$ to the center of $B_i$.
%Therefore, $UR_0(X)\leq\max_{i=1}^N r_i <2r$.
%
%This assertion already leads to the estimates that are
%worse than the estimates in our theorems. To get
%the assertion of the lemma one needs
%to improve the constant in the conclusion by a factor of $2$.
%The case when $X$ has dimension $0$ is obvious, so we can assume that
%$X$ is a connected Riemannian (multi)graph. 
Second, observe that the lemma can be reduced to its particular case, when $X$ is tree (endowed with Riemannian metric on each edge). Indeed,
by disconnecting some of the edges of $X$ from
one of their endpoints to destroy cycles in $X$ we can transform $X$ into a Riemannian tree
$Y$. 
We also have a (quotient) map $f:Y\longrightarrow X$, obtained by identifying
new vertices of $Y$ with the corresponding old ones.
The distances in $Y$ are not less than distances 
between the images of the same points in $X$. Therefore, each
metric ball with center $y$ in $Y$ is a subset of the 
metric ball in $X$ with the center $f(y)$ and the same radius. Therefore, the assumption of the lemma holds
for $Y$. On the other, if $Y$ can be covered by a metric ball with center $y$, then the metric ball in $X$ with the center $f(x)$ and the same radius will cover $X$.

%Now we are going to prove that (1) the conclusion of the lemma
%or $Y$ implies the conclusion of theorem for $X$; (2) the conclusion
%of the lemma holds for $Y$.  To see (1) observe that 
%if the conclusion of the lemma holds for $Y$, then $Y$ can be covered by a closed metric ball of radius $<r$ (in the metric of $Y$).
%Then the metric ball with the
%same center and radius but in the metric of $X$ will cover $X$.
Therefore, we can assume that $X$ is a Riemannian tree. Let $p$ be a center of $X$, that is, a point in $Y$ realizing the minimum, $R$,
of $\max_{q\in X}dist_X(p,q)$. Let $x\in Y$ denote one of the most distant points of $Y$ from $p$ (in the metric of $Y$). (So, $R$ is the radius of $X$.) 
As $X$ is contained in the metric ball of radius $R$ centered at $p$, we need to prove that $R<r$.
%Applying the argument above, we see that
%$Y$ is contained in a closed metric ball of radius $u<r$ centered at $x$.
%Let $u<2r$ be the minimal radius of such a ball.

Observe that the definition of $p$ implies that there exists
another point $x'\in Y$ such that $dist_Y(p, x')=dist_Y(p, x)$, and the
(unique) shortest path from $x$ to $x'$ passes through $p$ (as $Y$ is a tree). Therefore,
$dist(x,x')=2dist(p,x)=2R$. This implies that for each $t\leq R$ the length of $X\cap B(x,t)\geq 2t$, and therefore $R<r$.
%Hence,  $Y$ is contained in the metric ball of radius ${u\over 2}<r$ centered at $p$.
%One can map the interior $U(x)$ of $S(x)$ to $x$. 
%Now consider a finite covering of $X$ by
%$U(x_i)$ for some $x_1,\ldots , x_N$. Map
%$U(x_1)$ to $x_1$. Then, proceeding inductively, one
%maps $U(x_i)\setminus \bigcup_{j=1}^{i-1}U(x_j)$
%to $x_i$.

\end{proof}

\begin{lemma}
Assume that $B$ is a metric ball of radius $r$ centered at $x\in X^{n+1}$, $\epsilon\in (0,1)$. Assume that $|B|\leq cr^{n+1}(1-\epsilon)$  for some $c$.
Let $\lambda=r({\epsilon\over 3})^{1\over n+1}$. %$\epsilon_1={\epsilon\over 3}$
There exists a subset $A$ of the open interval $(\lambda,r)$ of positive measure such that
for each metric sphere $S$ centered at $x$ with radius $t\in A$, %$\rho\in (r, r(1-{1\over n+1}))$
$|S|< (n+1)ct^n(1-{1\over 3}\epsilon)$, and $S$ is a subpolyhedron of $X^{n+1}$.
\end{lemma}
\begin{proof}
Assume that the set of radii $>\lambda$ such that $|S|< (n+1)ct^n(1-{1\over 3}\epsilon)$ has measure zero. The coarea inequality implies $|B|\geq \int_\lambda^r c(n+1)t^n(1-\epsilon/3)dt> cr^{n+1}(1-\epsilon/3)-c\lambda^{n+1}=cr^{n+1}(1-{2\over 3}\epsilon)$, yielding a contradiction
with our assumption. To ensure that $S$ is a subpolyhedron, we apply Sard's theorem on each open simplex of $X^{n+1}$. (Recall,
that by distance function we actually mean a smooth approximation to the distance function.)
%The assertion 
%that $S$ can be chosen so that it is also a subpolyhedron would immediately 
%follows from Sard's theorem (applied to sets where the distance function to $x$ is smooth).
\end{proof}

\begin{definition}
A compact subpolyhedron $Y^{m-1}$ of $X^m$ is called $d$-separating if each connected
component of its complement $X^m\setminus Y^{m-1}$ can be covered by a metric ball of radius $\leq d$. %Let $HC_n^{(b)}(Y)$
%denote the infimum of $\sum_i r_i^n$ over all coverings of $Y$ by closed metric balls with radii $r_i$ such that all radii $r_i$ do not exceed $b$.
Denote
the infimum of $|Y|_{m-1}$ over all $d$-separating sets $Y$ in $X^m$ by $I_X(d, m-1)$.
If $\delta>0$
is a positive real number, a $d$-separating set $Y^{m-1}$ is called $\delta$-minimal if $|Y^{m-1}|\leq I_X(d,m-1)+\delta$.
\end{definition}

\begin{lemma}
Assume that $X^{n+1}$ is a Riemannian polyhedron of dimension $\leq n+1$  such that for some positive $r$, $\epsilon$ and $\tau$ each each $x\in X^{n+1}$
there exists $t(x)\in (\tau, r]$ %>\lambda$ 
such that the metric ball $B$ of radius $t(x)$ centered at $x$ satisfies
the inequality $|B|_{n+1}< {2t(x)^{n+1}\over (n+1)!}(1-\epsilon)$. Then there exists $\delta=\delta(n,\epsilon,\tau)$ such that
for every $\delta$-minimal $r$-separating set $Z$ and each $x\in X^{n+1}$ there exists $\rho \in (t(x)({\epsilon\over 3})^{1\over n+1},t(x))$ 
such that:
\par\noindent
(a) The metric ball $\beta$ in $X^{n+1}$ of radius $\rho$  centered at $x$ satisfies
$$|Z\bigcap\beta|_n< {2\rho^n\over n!}(1-{\epsilon\over 6});$$
\par\noindent
(b) If $x\in Z$, and $\alpha$ denotes the metric ball in $Z$ of radius $\rho$ centered at $x$, where $Z$ is endowed with the intrinsic metric, then the volume of $\alpha$ is less than ${2\rho^n\over n!}(1-{\epsilon\over 6})$.
\end{lemma}
%%%%%%%%%%%%%%%%%

\begin{proof} The distances between two points of $Z$ endowed with the inner metric cannot be less
than the distance between these points in the metric of $X^{n+1}$. 
Therefore, assertion (b) of the lemma follows from assertion (a) simply because
each metric ball in $Z$ endowed with the inner metric is contained in the metric ball in $X^{n+1}$ with the same center
and the same radius. Thus, it is sufficient to prove (a).

We apply the previous lemma to $B$. There exists a metric sphere $S$ of radius
$s\in (t(x)({\epsilon\over 3})^{1\over n+1},t(x))$ that is a subpolyhedron and satisfies $|S|<(1-{1\over 3}\epsilon) {2s^n\over n!}$. Let $\lambda=\tau({\epsilon\over 3})^{1\over n+1}$, $\delta= {\epsilon\lambda^n\over 3n!}$. %{\epsilon^2\over 18}{\lambda^{n+1}\over (n+1)!}$. 
Observe that $${2s^n\over n!}(1-{1\over
6}\epsilon)-|S|> {2s^n\over n!}{\epsilon\over 6} %>{\lambda({\epsilon\over 3})^{1\over n+1}\over n!}{\epsilon\over 6}
>\delta.\ \ \ \ (*).$$
%{1\over 2}({s^n\over n!}-|S|).$

The proof is by contradiction. Let $Z$ be a $\delta$-minimal $r$-separating set.
We assume that for some $x$ and
each $\rho\in (t(x)({\epsilon\over 3})^{1\over n+1},t(x))$, the metric ball $\beta$ of radius $\rho$ centered at $x$ does not
satisfy the inequality in part (a) of the lemma. In particular, this is true for the metric ball $\beta$ of radius $s$ bounded by $S$: $|Z\bigcap\beta|\geq {2s^n\over n!}(1-{\epsilon\over 6})$,
We are going to modify $Z$ to obtain another $r$-separating set $Z'$ with 
volume less than $|Z|-\delta$ arriving to a contradiction.
%We apply the previous lemma to $B$. There exists a metric sphere $S$ of radius
%$s\in (0,t(x)]$ that is a subpolyhedron and satisfies $|S|< {s^n\over n!}$.
%Our assumption implies that for the metric ball $\beta$ bounded by $S$ %$|Z\bigcap\beta|\geq {s^n\over n!}(1-{\epsilon\over 6})$.
%start from choosing a sphere $S$ centered at $x$ as in the previous lemma. Using also Sard's theorem we can be sure that
%$S$ is a subpolyhedron of $X^{n+1}$. Denote its radius by $t$.
%So, $|S|\leq t^n,$ and $\beta$ is in the interior
%of the metric ball bounded by $S$.

To construct $Z'$ we remove from $Z$ all points inside the metric ball bounded by $S$, and take the union of the resulting
set $Z_1$ with $S$. It is obvious that $Z'$ is $r$-separating. Indeed, all components of $X^{n+1}\setminus Z$ outside of $S$
became smaller or unchanged when we replace $Z$ by $Z'$, and the ``new" component or components of $X^{n+1}\setminus Z'$ inside
of $S$ can clearly be covered by the metric ball of radius $r$ centered at $x$. 

It follows from formula (*) that $|Z'|< |Z|-\delta$.

%Now we are going to estimate $|Z_1|=|Z|-|Z\bigcap\beta|+|S|\leq |Z|-\epsilon_n(1-{1\over n+1})^nr^n+(n+1)\epsilon_{n+1}r^n= %|Z|-2^{-n-1}(n+1)^{-n}\sigma r^n< |Z|-\delta$, if $\delta<2^{-(n+1)}(n+1)^{-n}r^{-n}\sigma$.

\end{proof}

%\begin{corollary}
%Under the assumptions of the previous lemma the volume of each metric ball of radius $\rho=(1-{1\over n+1})r$ in the Riemannian subpolyhedron $Z$ does not exceed $\epsilon_n\rho^n$.
%\end{corollary}
%\begin{proof}
%As the extrinsic distance never exceed the intrinsic distance, any metric ball with respect to the intrinsic distance in $Z$
%is contained in the intersection of $Z$ with the concentric extrinsic metric ball in $X^{n+1}$ of the same radius. 
%In the situation of the corollary, Lemma 3.4 yields the desired upper
%bound for the volume of this intersection, and, therefore, for the volume of the metric ball in $Z$.
%\end{proof}

\begin{lemma}
Assume that $Y^n$ is a $d$-separating subpolyhedron in $X^{n+1}$ such that for some $d$ $UR_{n-1}(Y^n)\leq d$. Then $UR_n(X^{n+1})\leq d$.
\end{lemma}
\begin{proof}
{\it First proof.} Let $\{U_\alpha\}_{\alpha\in A}$  be the set of all connected components of $X\setminus Y$. Observe that for each $\alpha$, $\partial U_\alpha\subset Y$. Consider the map $f$  of $Y$ to a $(n-1)$-dimensional
simplicial complex $K$ such that for each $x\in K$, $f^{-1}(x)$ can be covered by a metric ball $b(x)$ of radius $d+\delta$ in the intrinsic metric of $Y^n$, where $\delta$ is
arbitrarily small. Observe that the metric ball in $X^{n+1}$ with the same center
and radius will contain $b(x)$, and, therefore, $f^{-1}(x)$.
For each $\alpha$
take a copy $CK_\alpha$ of the cone $CK$ over $K$. Using the version of Tietze extension theorem for maps into contractible
simplicial complexes, we can extend the restriction of $f$ to $\partial U_\alpha$ to a continuous
map $g$ of the closure of $U_\alpha$ to $CK_\alpha$ (compare with a similar argument in
section 6.1 of [LLNR]). We would like to change this map so that the images of all points of $\partial U$ would remain unchanged,
and all points in $\bar U_\alpha\setminus \partial U_\alpha$ will be mapped to $CK_\alpha\setminus K_\alpha$ (that is, to the
interior of the cone). For this purpose endow each top-dimensional simplex in $CK_\alpha$ by the metric of
the Euclidean regular simplex with side length $d$. For each $x\in\bar U_\alpha$ $g(x)$ will be either the tip of the cone,
or a point on the unique generator of the cone passing through $x$ (which is a straight line segment with one end at $K_\alpha$ and another end at
the tip of the cone). If $g(x)$ is the tip of the cone, the new map $h(x)=g(x)$. Otherwise, let $\phi(x)$ denote the distance
from $g(x)$ to the tip of the cone. Now move $g(x)$ towards the tip of the cone by $\min\{\phi(x), dist(x,\partial U)\}$.
Now glue all copies of $CK_\alpha$ into one $n$-dimensional simplicial complex $L$ by identifying all copies of $K_\alpha$ at the boundaries into one
copy of $K$.

The resulting map $h:X\longrightarrow L$ is a continuous map. By construction, its restriction to $Y$ coincides with $f$,
and for each point $x\in Y$, $f^{-1}(f(x))$ is in $Y$. For each $x\in U_\alpha$, $f(x)\in CK_\alpha\setminus K$, and $f^{-1}(f(x))\in U_\alpha$. In both cases $f^{-1}(f(x))$ can be covered by a ball of radius $d+\delta$.

\par\noindent
{\it Second proof.} After reading the first version of this paper ([N]) Roman Karasev suggested the following very simple proof of
Lemma 2.5.
His proof uses the definition of $UR_n(X)$ as the infimum of $r$ such that there exists a cover of $X$
by open sets with radii $\leq r$ with multiplicity of the covering $\leq n+1$. Start with an
open covering of $Y^n$ of multiplicity $\leq n$ with radii of the sets in the intrinsic metric $\leq d$. 
Then we convert this covering of $Y^n$ into an open covering of a very small open neighbourhood of $Y^n$ in $X^{n+1}$ without
increasing the multiplicity and increasing the radii by not more than an arbitrarily small amount. Finally, we add all open sets $U_\alpha$ to the covering increasing its multiplicity by $1$.
\end{proof}

\begin{proposition}

Let $r,\epsilon, \tau<r$ be positive real numbers, and $X^n$ is a compact $n$-dimensional Riemannian polyhedron. Assume that for each $x\in X^n$
there exists $t(x)\in (\tau, r]$ such that the metric ball $B$ in $X^n$ of radius $t(x)$ centered at $x$ satisfies
$|B|< {2t(x)^n\over n!}(1-\epsilon)$.% where $\epsilon_n$ were defined at the beginning of this section.
Then $UR_{n-1}(X)<r$.%{3\over n+2}r$.
\end{proposition}

\begin{proof}
%Now the generalization of Theorem 2.1 for Riemannian polyhedra 
We are going to prove this proposition using the induction with respect 
to the dimension $n$. Lemma 2.1 is the base of induction.
To prove the induction step assume that the theorem is true for $n$. To prove it for $n+1$ choose a sufficiently small positive $\delta$ (as in Lemma 2.4) and consider a $\delta$-minimal $r$-separating $n$-dimensional subpolyhedron $Z$ of $X^{n+1}$.
Lemma 2.4 immediately implies that Riemannian subpolyhedron $Z$ satisfies the assumptions of the proposition the following values of the parameters $n,r,\epsilon,\tau,t(x)$:  We take $n$ and $r$ as in the text of the proposition, ${\epsilon\over 6}$ as a new value of $\epsilon$,  and $\tau({\epsilon\over 3})^{1\over n+1}$ as a new value of $\tau$. Finally, we take $\rho\in (\tau({\epsilon\over 3})^{1\over n+1}, r]$ provided by Lemma 2.4 as the value of $t(x)$. % and $\rho<t(x)$.

The induction assumption implies that $UR_{n-1}(Z)\leq \rho< r$ in the intrinsic metric on $Z$. %\leq (1-{1\over n+3})r{3\over n+2}={3\over n+3}r$.
The same inequality will automatically be true for the ``shorter" extrinsic metric.
Now the induction step follows from Lemma 2.5 applied for $Y^n=Z$, $d=r$.
\end{proof}

Now we are going to establish Theorem 1.1 for the class of all compact Riemannian polyhedra (and not only Riemannian manifolds):
\begin{theorem}
Let $M^n$ be a compact Riemannian polyhedron (for example, a closed Riemannian manifold), and $r>0$ a real number. Assume that for every $x\in M^n$ there exists $t=t(x)\in (0,r]$ such that
the volume of the metric ball of radius $t$ centered at $x$ is
less than ${2t^n\over n!}$.  Then $UR_{n-1}(M^n)< r.$
\end{theorem}
\begin{proof}
%Now we would like to deduce Theorem 1.1 from Theorem 2.6 in the case when $X^n=M^n$ is a Riemannian manifold. 
We are going to deduce this theorem from Proposition 2.6 by proving that its assumption can be relaxed in two ways.

First, we would like to demonstrate that the assumption of existence of $\epsilon>0$ such that for each $x$ there exists $t\in [\tau,r]$ such that the ball $B=B(x,t)$ satisfies $\vert B\vert<{2t(x)^n\over n!}(1-\epsilon)$
is equivalent to the assumption that for each $x$ there exists $t\in [\tau,r]$ such that $\vert B\vert<{2t(x)^n\over n!}.$
Observe that ${|B(x,R)|\over R^n}$ is a continuous function of the center $x\in M^n$ of the ball $B(x,R)$ and its radius $R$.
Note that $\rho(x)=\min_{R\in [\tau,r]} {|B(x,R)|\over R^n}$ is a continuous function of $x$ such that its value at every point is strictly less than ${2\over n!}$. Hence, the maximum of $\rho(x)$ over all $x\in M^n$ will be attained at some point, and will be strictly less than ${2\over n!}$. Now one can choose $\epsilon$ as $0.5({2\over n!}-\max_{x\in M^n}\rho(x))$.

Second, we are going
to demonstrate that one does not need the condition that there exists $\tau>0$ such for each $x$  and some $t$ as in Proposition 2.6 $t\geq \tau$, 
as this condition automatically holds. 
For all $x\in M^n$ formally define $t(x)$ as $\sup_{t\in (0,r]}\{t\vert {|B(x,t)|\over t^n}< {2\over n!}\}$. The assumption of the theorem is that the set of $t$ is non-empty, and, therefore, $t(x)$ is defined 
 and positive for all $x$. Observe that $t(x)$ is lower-semicontinuous, i.e. $t(x)\leq \lim\inf_{y\longrightarrow x} t(y)$. Indeed, as ${|B(x,t)|\over t^n}$ is continuous, if ${|B(x,t)|\over t^n}<{2\over n!}$ for some $t$, then the same inequality will be true for all $y$ sufficiently
 close to $x$. Therefore, for each positive $\delta$ the inequality $t(y)\geq t(x)-\delta$ holds for all $y$ sufficiently close to $x$. This observation immediately implies the lower-semicontinuity of $t$. Hence, $t(x)$ attains its positive minimum on $M^n$ which can be chosen as $\tau$. Now the theorem follows from Proposition 2.6.

%Note that for each closed Riemannian manifold $M^n$ the volume of metric balls of with a very small radius $t$
%behaves as $\omega_nt^n$, where $\omega_n$ is the volume of the unit ball of radius $1$ in the Euclidean space. The error can be majorized
%in terms of the $C^1$-norm of the curvature tensor of $M^n$.
%As $\omega_n>{2\over n!}$, there exists $\tau>0$ such that for each $t\leq\tau$ and each $x\in M^n$ the volume of the metric ball of radius $t$
%centered at $x$ is greater than ${2t^n\over n!}$. Thus, $t(x)$ is uniformly bounded below by a positive constant $\tau$.
%
%First, we would like to demonstrate that one can replace the condition of existence of some $t$ such that $\vert B\vert<{2t(x)^n\over n!}(1-\delta)$ by $\vert B\vert<{2t(x)^n\over n!}.$
%Observe that ${|B(x,R)|\over R^n}$ is a continuous function of the center $x\in M^n$ of the ball $B(x,R)$ and its radius $R$.
%Note that $R(x)=\min_{R\in [\tau,r]} {|B(x,t)|\over t^n}$ is a continuous function of $x$ such that its value at every point is strictly less than ${2\over n!}$. Hence, the maximum of $f(x)$ over all $x\in M^n$ will be attained at some point, and will be strictly less than ${2\over n!}$. Now one can choose $\delta$ as $0.5({2\over n!}-\max_{x\in M^n}f(x))$.
%Theorem 1.1 immediately follows from Theorem 2.6.
\end{proof}
\vskip 0.3truecm

\par\noindent
{\bf Proof of the inequality (4) in section 1.2}: Finally, I am going to present an elementary proof of the inequality $sys_1(X^n)\leq 2UR_{n-1}(X^n)$ for essential polyhedral
length spaces $X^n$. I learned the idea of this proof from Roman Karasev. This proof is based on ideas from [S].

Recall that $UR_{n-1}(X^n)$
can be defined as the lower bound of $r$ such that there exists a covering of $X^n$ by connected open sets with radii $\leq r$ with multiplicity $\leq n$.

Assume that $sys_1(X^n)>2r$, where $r=UR_{n-1}(X^n)$. Choose a covering of $X^n$ of multiplicity $\leq n$ by connected open sets $U_\alpha$  with radii
$\leq r+\delta$, where $\delta<{1/3}(sys_1(M^n)-2r)$. Consider a closed curve $\gamma$
in some $U_\alpha$. %, or the union of two sets $U_\alpha$, $U_\beta$, or the union
%of three sets $U_\alpha$, $U_\beta$, $U_\gamma$. 
Let $p$ be the center of a ball of
radius $r+\delta$ covering $U_\alpha$. 
%Observe that the distance from $p$ to any point $\gamma(t)$ on $\gamma$ does not exceed $3r+3\delta$. 
We can homotope $\gamma$ into a concatenation of many thin triangles $p\gamma(t_i)\gamma(t_{i+1})$, where the length of the arc of $\gamma$ between
$\gamma(t_i)$ and $\gamma(t_{i+1})$ does not exceed $\delta$, and two other sides are minimizing geodesics.
The length of each of these triangles is less than $sys_1(X^n)$. Therefore, these
triangles are contractible, and so is $\gamma$. Thus, the inclusion homomorphisms $\pi_1(U_\alpha)\longrightarrow\pi_1(X^n)$
are trivial, and each $U_\alpha$ lifts to a collection of disjoint open sets $(\widetilde U_\alpha)_g\subset \tilde X^n$, where $g$ runs over $\pi_1(X^n)$, and $\widetilde X^n$ denotes the universal covering of $X^n$ endowed with the pullback metric.

Consider the nerves $N$ of the covering $\{U_\alpha\}$ of $X^n$, and $\widetilde N$ of the covering $\{(\widetilde U_\alpha)_g\}$ of $\widetilde X^n$. It is easy to see that there exists a commutative square with horizontal sides $\widetilde X^n\longrightarrow \widetilde N$ and $X^n\longrightarrow N$, where vertical sides $\widetilde X^n\longrightarrow X^n$, and $\widetilde N\longrightarrow N$ are the universal covering maps.
This easily implies that the map $X^n\longrightarrow N$ induces an injective homomorphism $\pi_1(X^n)\longrightarrow \pi_1(N)$.
%This immediately implies that
%the homomorphism of the fundamental groups induced by the map of $X^n$ to the nerve
%of the covering is injective. 
As this homomorphism is obviously surjective, it is an isomorphism. Thus, the classifying map $X^n\longrightarrow K(\pi_1(M^n),1)$
factors through the nerve $N$ that has dimension $\leq n-1$. (Recall that all maps
$X^n\longrightarrow K(\pi_1(X^n),1)$ that induce the isomorphism of fundamental groups are homotopic.)
Therefore, $X^n$ is not essential. Equivalently, if $X^n$ is essential, then $sys_1(X^n)\leq 2r$.

\section{ Hausdorff content}

Similar ideas can be applied to majorize $UR_{m-1}$ of a compact or  even a boundedly compact metric space $X$ in terms of the Hausdorff content, $\HC_m$,
of metric balls in $X$. The key is the coarea inequality proven in [LLNR]. (Note that in this section we are no longer assuming that $X$ is a Riemannian polyhedron or even a length space.) 

Recall that a metric space is called boundedly compact if all its closed and bounded subsets are compact. Given a boundedly compact metric space $X$, its bounded subset $A$, and a positive real $m$, one defines the $m$-dimensional Hausdorff content, $HC_m(A)$, of $A$ as $\inf \sum_i r^m$, where the infimum is taken over all
coverings of $A$ by closed metric balls $\beta_i$ of radii $r_i$ in $X$. If $A$ is empty, then it can be covered by the empty set of metric balls, and $\HC_m(A)=0$. 
%Note that we also can consider $A$ as a metric space with the metric coming from $X$. Metric balls of radius $r$ in $A$ are just intersections of $A$ with metric balls of radius $r$ in $X$ centered at a point of $A$. If one regards $A$ as a subset of itself regarded as the ambient metric space, one obtains
%{\it the intrinsic} Hausdorff content $HC_m(A)_i$. Clearly, $HC_m(A)_i\geq HC_m(A)$.

For convenience of the reader we present a version of the coarea inequality for Hausdorff content proven in [LLNR]:

\begin{lemma} ([LLNR])
Let $m, r_1, r_2$ be real numbers such that $0\leq r_1<r_2, m>0$, $X$ a metric space, $Y$ a subset of $X$, $x$ a point in $X$. For $j\in\{1, 2\}$ let $B_j$ be closed metric balls of radius $r_j$ centered at $x$, $A$ the annulus $B_2\setminus B_1$, and $S_s,\ s\in (r_1, r_2)$,
metric spheres of radius $s$ centered at $x$. Then there exists $s\in (r_1,r_2)$
such that $\HC_{m-1}(S_s\cap Y)\leq {2\over r_2-r_1}\HC_m(A\cap Y)$.
In particular, when $r_1=0$, we see that for each $r$ and metric ball $B$ of radius $r$ centered at $x$
there exists a positive $s<r$ such that
$\HC_{m-1}(S_s\cap Y)\leq {2\over r}\HC_m(B\cap Y)$.
\end{lemma}

\begin{proof}
For an arbitrarily small $\epsilon>0$ choose a covering of $A\cap Y$ by a countable collection of closed metric
balls $\beta_i\subset X$ with radii $r_i$ such that $\sum_i r_i^m\leq\HC_m(A\cap Y)+\epsilon$. Observe
that $\HC_{m-1}(S_s\cap Y)\leq \HC_{m-1}^*(S_s\cap Y)$ defined as $\sum_{i\in I_s}r_i^{m-1}$, where
$I_s$ is the set of indices $i$ such that $\beta_i\cap S_s\cap Y\not= \emptyset$. We are going to
prove that for some $s$ $\HC_{m-1}^*(S_s\cap Y)\leq {2\over r_2-r_1}\sum_i r_i^m$, which implies the lemma. For this purpose it is sufficient to prove that $\int_{r_1}^{r_2}\HC_{m-1}^*(S_s\cap Y)ds\leq 2\sum_ir_i^m.$
For each $i$ and $s$ define $\chi_i(s)$ as $1$, if $B_i\cap S_s\cap Y\not=\emptyset$, and $0$
otherwise. Using this notation $\int_{r_1}^{r_2}\HC_{m-1}^*(S_s\cap Y)ds=\int_{r_1}^{r_2}\sum_i 
r_i^{m-1}\chi_i(s)ds=\sum_ir_i^{m-1}\int_{r_1}^{r_2}\chi_i(s)ds\leq\sum_ir_i^{m-1}(2r_i)=2\sum_i r_i^m$.
\end{proof}

\noindent
{\bf Remark.} The constant $2$ in the right hand side of the coarea inequality is optimal. Indeed, let $m=1$, $r_1=0$, $r_2=1$, $X=Y=[0,1]$, $x=0$, and $A=(0,1]$.
$HC_1(A\cap Y)={1\over 2}$, as $A$ can be covered by a ball of radius $0.5$ centered at $0.5$, and $HC_0$ of the intersection of $Y$ with any geodesic sphere of radius in $(0,1)$ centered at $x=0$ is $1$.
%Let $X$ be a compact metric space. For each non-negative integer $n$ define Urysohn radius of $X$, $UR_n(X)$, as the infimum
%of $r$ such that there exists  continuous map $f:X\longrightarrow K^n$ from $X$ to a $n$-dimensional simplicial complex $K^n$ such that for each $s\in K^n$ $f^{-1}(s)$ can be covered by a metric ball of radius $r$. It is clear that $UW_n(X)\leq 2UR_n(X)$.
%%Also, if $X$ is $n$-essential, then $sys_1(X)\leq 6UR_n(X)$ (see section 8 for the details).

\begin{lemma}
Let $X$ be a metric space, $Y$ a subset of $X$.
Assume that for every metric ball $B$ of radius $r$ in $X$ $HC_1(B\cap Y)< {r\over 2}$.
Then $UR_0(Y)< r$.
\end{lemma}

\begin{proof}
We are going to prove that each connected component $C$ of $Y$ is contained in the interior of a closed metric ball $B(C)=B(x,\rho)$ of radius
$\rho< r$ with a center $x=x(C)\in C$. (In fact, we will see that we can choose any point $x\in C$ as $x(C)$.)
If so, we can map $Y$ into the set of centers of balls $B(C)$ by sending all points of $C$ to the center $x(C)$ of $B(C)$.

Given $x\in C$ apply the coarea inequality (Lemma 3.1) to the closed ball $B(x,r)\subset X$ of radius
$r$ centered at $x$ (minus $x$) regarded as the annulus with radii $r_1=0$ and $r_2=r$. We are going to obtain a geodesic sphere $S$ 
centered at $x$ of a positive radius $\rho< r$ such that $HC_0(S\cap Y)\leq {2\over r}HC_1(B(x,r)\cap Y)<1.$
Note that $HC_0(S\cap Y)$ is just the minimal number of metric balls in $X$ required to
cover $S\cap Y$, and so it is either equal to $1$, if $S\cap Y$ is non-empty, or to $0$, if $S\cap Y$ is empty. Therefore, we conclude that $HC_0(S\cap Y)=0$, and $S\cap Y$ is empty.
Therefore, $B(x,\rho)\cap Y$ coincides with the intersection of $Y$ with the open metric ball of radius $\rho$ centered at $x$, and is the union of 
some collection of connected components of $Y$, one of which coincides with $C$. Now we can define  $B(C)$ as $B(x,\rho)$.
%a closed metric ball $B(x,s)$ of radius 
%$s=s(x)<\rho<r$ centered at $x$, which is a union of connected components of $X$. As $B(x,s)$ coincides with the open metric ball of radius $\rho$ centered at $x$,
%It is easy to see that for two different points $x_1,x_2\in X$ the corresponding balls $B(x_1,s(x_1))$ and $B(x_2,s(x_2))$
%either coincide or are disjoint.
\end{proof}
\begin{definition}
Given a subset $W$ of $X$, positive real $n$, and $\delta>0$, a collection
of metric balls $B_i$ in $X$ with radii $r_i$ is called a $(n,\delta)$-optimal covering of $W$, if they cover $W$, and $\sum_i r_i^n\leq HC_n(W)+\delta.$
\end{definition}

\begin{definition} Let $X$ be a metric space, $Y$ its compact subset.
A compact subset $Z$ of $Y$ is called $d$-separating for $Y$ if each connected
component of its complement $Y\setminus Z$ can be covered by a metric ball in $X$ of radius $\leq d$. Let $HC_n^{(b)}(Z)$
denote the infimum of $\sum_i r_i^n$ over all coverings of $Z$ by closed metric balls with radii $r_i\leq b$ in $X$. 
Denote the infimum of $HC_n^{(b)}(Z)$ over all $d$-separating sets $Z$ by $I_Y(d,b,n)$.
If $\delta>0$
is a positive real number, a $d$-separating set $Z$ is called $(b,n,\delta)$-minimal if $HC_n^{(b)}(Z)\leq I_Y(d,b,n)+\delta$.
\end{definition}

Using $HC_n^{(b)}(Y)$ instead of $HC_n(Y)$ here is another simple but beautiful idea of Papasoglu from [P] designed to
overcome non-additivity of Hausdorff content (and strongly reminiscent of ideas earlier used for the same purpose in [LLNR]).
%Define $\epsilon_1={1\over 3}$, and for each $n\geq 1$ 
%$\epsilon_{n+1}={\epsilon_n\over 2(n+3)}\exp(-2)<{\epsilon_n\over 2(n+3)}(1-{2\over n+3})^n$.
%Obviously, $\epsilon_{n+1} =2(2\exp(2))^{-n}{1\over (n+3)!},$ and it is not difficult to see that $\epsilon_{n+1}^{1\over n+1}<{1\over 3(n+3)}$.

\begin{lemma} Let $X$ be a metric space, $Y$ its compact subset, $r, \mu$ positive real numbers, $n\geq 1$ an integer.
Assume that for each closed metric ball $B$ of radius $r$ in $X$, $HC_{n+1}(B\cap Y)<({r\over 4(n+1)})^{n+1}-\mu$. Let
$\mu_1={r\over 4(n+1)}-(({r\over 4(n+1)})^n-4(n+1){\mu\over r})^{1\over n}$.
Assume that
$Z$ is a $({r\over 4(n+1)}-\mu_1,n,4(n+1){\mu\over r})$-minimal %$r(1-{1\over n+3)})$-
$r$-separating set for $Y$.
%for a sufficiently small positive $\delta$.%=\delta(n,r)$.
Then for each ball $\beta$ of radius $\rho=r(1-{1\over n+1})$ in $X$ 
$$HC_n(Z\bigcap\beta)< ({\rho\over 4n})^n.$$
\end{lemma}
%%%%%%%%%%%%%%%%%

\begin{proof}
The proof is by contradiction. We assume that there exists a ball $\beta$ or radius $\rho$ centered at a point $x$ that does not
satisfy the above inequality. We are going to modify $Z$ to obtain another $r$-separating set $Z'$ with a significantly lower
$HC_n^{({r\over 4(n+1)}-\mu_1)}$ than $Z$, obtaining a contradiction that proves the lemma. 

Using Lemma 3.1 one can choose a sphere $S$ centered at $x$
of radius in the interval $(r(1-{1\over 2(n+1)}), r)$
such that $HC_n(S\cap Y)< ({r\over 4(n+1)})^n-4(n+1){\mu\over r}=({r\over 4(n+1)}-\mu_1)^n.$
Therefore, 
%As for $n\geq 1$, $4(n+1)\epsilon_{n+1}={1\over (4(n+1))^n},$
$HC_n^{({r\over 4(n+1)}-\mu_1)}(S\cap Y)=HC_n(S\cap Y)$. 
%Also, the radius of $S$ does not
%exceed $r$, %(1-{1\over n+3})$, 
%and is not less than  $r(1-{1\over 2(n+1)})$.

To construct $Z'$ we remove from $Z$ all points inside the ball bounded by $S$, and take the union of the resulting
set $Z_1$ with $S\cap Y$. It is obvious that $Z'$ is $r$-separating.
%$r(1-{1\over n+3})$-separating.

Now we are going to estimate $HC_n^{({r\over 4(n+1)}-\mu_1)}(Z_1)$. First, note that none of the balls of radius $\leq {r\over 4(n+1)}-\mu_1$ in $X$
used to cover $Z_1$
in a nearly optimal way can intersect the closed ball $B'$ of radius $r(1-{1\over 2(n+1)})$ centered at $x$. On the other hand, every metric ball of radius $\leq {r\over 4(n+1)}-\mu_1$ that has a non-empty intersection
with $\beta$ is contained in $B'$. Therefore, $HC_n^{({r\over 4(n+1)}-\mu_1)}(Z)\geq HC_n^{({r\over 4(n+1)}-\mu_1)}(Z_1) + HC_n^{({r\over 4(n+1)}-\mu_1)}(Z\cap\beta)\geq HC_n^{({r\over 4(n+1)}-\mu_1)}(Z_1) + HC_n(Z\cap\beta)$.
Hence,
$HC_n^{({r\over 4(n+1)}-\mu_1)}(Z_1)\leq HC_n^{({r\over 4(n+1)}-\mu_1)}(Z)-HC_n(Z\cap\beta)$, and 
$HC_n^{({r\over 4(n+1)}-\mu_1)}(Z')\leq HC_n^{({r\over 4(n+1)}-\mu_1)}(Z)-HC_n(Z\cap\beta)+HC_n^{({r\over 4(n+1)}-\mu_1)}(S\cap Y)< HC_n^{({r\over 4(n+1)}-\mu_1)}(Z) -({r\over 4n})^n(1-{1\over n+1})^n+(({r\over 4(n+1)})^n-4(n+1){\mu\over r})= HC_n^{({r\over 4(n+1)}-\mu_1)}(Z)-4(n+1){\mu\over r}$.
%for a sufficiently small positive $\delta=\delta(n,r).$

\end{proof}
\begin{lemma} Let $X$ be a metric space, $Y$ a closed subset of $X$, $r$ a positive real number, $Z$ a $r$-separating set for $Y$. %Consider $Y$ and $Z$ as metric spaces with the induced metric.
Assume that $Z$ is compact, and $UR_{n-1}(Z)\leq r$. Then $UR_n(Y)\leq r$.
\end{lemma}
\begin{proof}
%The only difference of this lemma from Lemma 2.5 is that we do not assume that $X$ is an $n$-dimensional Riemannian polyhedron, and $Y$ its
%subpolyhedron of codimension $1$.
%Yet 
Either of the two proofs of Lemma 2.5 can be used with only minor modifications to prove this lemma.
For example, if $UR_{n-1}(Z)\leq r$, then for an arbitrary positive $\delta$,
there exits a finite collection of open sets of radius $< r+\delta$ in $Z$ that covers $Z$ and has multiplicity $\leq n$.
It is easy to demonstrate that one can modify this collection so that the open sets in $Z$ become restrictions of open subsets of $X$ with radius $<r+2\delta$, and the multiplicity of the resulting collection of open sets in $X$ does not exceed $n$. Taking the intersections of these open sets with $Y$, we obtain a collection of open subsets of $Y$ of
multiplicity $\leq n$ and radius  $< r+2\delta$.
After adding all connected components of $Y\setminus Z$ to this collection we obtain a covering of $Y$ by open in $Y$ sets of radius $<r+2\delta$ such that the multiplicity of the resulting collection of open sets does not exceed $ n+1$.
\end{proof}

\begin{theorem}
Let $r$ be a positive number, $n$ a positive integer, $X$ a metric space, $Y$ a compact subset of $X$ such that for each metric ball $B$ of radius $r$,
$HC_n(B\cap Y)< ({r\over 4n})^n.$ %where $\epsilon_n$ were defined at the beginning of this section.
Then $UR_{n-1}(Y)< r.$ %{3\over n+2}r$.
\end{theorem}
\begin{proof}
We are going to use the induction with respect $n$. Lemma 3.2 provides the base of induction.
To prove the induction step assume that the theorem is true for $n$. To prove it for $n+1$ observe that the assumptions of the theorem and the compactness of $Y$ imply the existence of a positive $\mu$ such that for each metric ball $B$ of radius $r$, $HC_{n+1}(B\cap Y)<({r\over 4(n+1)})^{n+1}-\mu$.
%choose a sufficiently small positive $\delta$ (as in Lemma 3.6) and consider a $({1\over 4(n+1)}r,n+1,\delta)$-minimal %$r(1-{1\over n+3})$
%$r$-separating set in $Y$. 
Define $\mu_1$ as in Lemma 3.5, and consider a $({r\over 4(n+1)}-\mu_1,n,4(n+1){\mu\over r})$-separating set $Z$.
Lemma 3.5 implies that $Z$ satisfies the conditions of the present theorem for $Y$ with parameters $n$ and $\rho=r(1-{1\over n+1})$ (instead of $r$).
The induction assumption implies that $UR_{n-1}(Z)< \rho<r$. %\leq (1-{1\over n+3})r{3\over n+2}={3\over n+3}r$.
Now the induction step follows from Lemma 3.6. % applied for $Y$, $Z$, and $d=r$.
%\begin{lemma}
%Assume that $Y$ is a $d$-separating set in $X$ such that for some $d$ $UR_{n-1}(Y)\leq d$. Then $UR_n(X)\leq d$.
%\end{lemma}
%begin{proof}
%Each connect component $U$ of $X\setminus Z$ has boundary in $Z$. By assumption, it can be mapped 
%\end{proof}
\end{proof}
\par\noindent
{\it Proof of Theorem 1.4.} The first part of Theorem 1.4 immediately follows from Theorem 3.7 applied for $Y=X$. 
To prove the second part it is sufficient
to take $r=({\HC_n(X)\over \epsilon_n})^{1\over n}+\delta$ for an arbitrarily small $\delta>0$, apply the first part, and pass to the limit, when $\delta\longrightarrow 0$. 

To prove the third part observe that the main difficulty in the non-compact case is that every $r$-separating set in $X$ might have
infinite Hausdorff content. In this case there will be no (almost) minimal $r$-separating sets, and (key) Lemma 3.5 cannot be applied.
Instead, one can use another trick from [P]: Let $\rho>0$
be a real number.
One chooses a point $x$ of a boundedly compact $X$ and covers $X$ by two overlapping sets of closed annuli centered at $x$. One family of annuli
involves radii in the intervals $[8(i-1)\rho, 8i\rho]$ for all positive integer $i$, another $[(8(i-1)+4)\rho, (8i+4)\rho]$. The idea is that the union
of almost minimal $\rho$-separating subsets for all these annuli will be a $\rho$-separating family for $X$. One can even remove the parts of  almost minimal
separating sets in all annuli that bound a domain only together with a non-empty subset of the boundary of the annulus. Indeed, all points in such domains with ``destroyed" boundary
will be $2\rho$-close to the boundary of the annulus, and, therefore, $2\rho$-close to the central sphere of an overlapping 
annulus $A$ in the second family. Therefore,
they will be $2\rho$-far from the boundary of $A$ and will be in domains in $A$ of radius $\leq \rho$ such that their closures do not intersect $\partial A$.

Now we can take each annulus (in one of the two collections) as $Y$ in Lemma 3.5, and apply Lemma 3.5. Remove the intersection of
the separating set $Z$ (as in Lemma 3.5) with the boundary of the annulus. This remaining part will also satisfy the conclusion of Lemma 3.5. Now take the union of the remaining parts of $Z$ over all annuli. As we observed, we will obtain an $r$-separating set
in the whole metric space $X$.
Each ball of radius $\rho$ (as in Lemma 3.5) will intersect almost minimal $r$-separating sets, $Z$, coming from at most two  annuli.
%Consider the remaining parts of almost minimal separating
%sets in annuli satisfy the assumptions and conclusion of Lemma 3.6 (separately for each annulus). As the result, 
Therefore, the conclusion of Lemma 3.5 will be almost true for 
{\it the union} $U$ of all these separating sets: One will only need the extra factor of $2$ in the right hand side of the inequality in the conclusion of Lemma 3.5. 
Thus, we obtain an analogue of Lemma 3.5 for non-compact boundedly compact metric spaces for $U$ instead of $Z$ with
the extra factor of $2$ in the right-hand side of the inequality.

Going through the rest of the proof, we see that
%We loose just a factor of $2$ in the conclusion of Lemma 3.6
%as we need to combine the union of two families of disjoint separating subsets.
this leads to the appearance of the extra factor of $2$ in the denominator of the right-hand side of the inequality in Theorem 1.4, part 3.

\par\noindent
{\bf Remark.} We do not see how to adapt this proof of Theorem 1.4(1) to prove its version, where the assumption that each metric ball $B$ of radius $r$ satisfies the inequality
$HC_n(B)< ({r\over 4n})^n$ is replaced by a weaker 
assumption that for each $r$ there exists $\rho\in (0,r)$ such that $HC_n(B)\leq c(n)\rho^n$, where one is allowed to 
choose any positive constant $c(n)$. The reason is that one uses $HC_n^{({r\over 4n+4}-\mu_1)}$ in the proof of Lemma 3.5, and it is not clear what is the correct replacement
of this quantity if $r$ is allowed to be variable. So, we do not know how to prove a Hausdorff content analog of Theorem 1.1 where the radii of ``small" balls can be variable.

{\bf Acknowledgements.} This work was partially supported by NSERC Discovery grant of the author. 

I would like to thank Roman Karasev who noticed several typos in the first version of the paper and suggested a simplification of the proof of Lemma 2.5, and also for communicating to me the inequality $sys_1(X^n)\leq 2UR_{n-1}(X^n)$ for essential polyhedral length spaces. I would like to thank Anton Petrunin who asked me if the constant in original version of Lemma 2.1 can be improved. This led me to a version with a better constant and helped to somewhat improve the constants in several theorems. 
%To prove Theorem 1.1 it is sufficient just to choose $t(x)=r$ in Theorem 2.6.
%Now observing that $\lim_{\sigma\longrightarrow 0}\epsilon_n={1\over n^n}$ we obtain Theorem 0.1 (and, therefore, Theorem 0.2)
%for Riemannian polyhedra, and in particular, closed Riemannian manifolds.


\begin{thebibliography}{}

\bibitem[BZ]{BZ} Burago, Yu. D., Zalgaller, V., Geometric inequalities, Springer, 1988

\bibitem[Ga]{Ga} Gaffney, M, ``The conservation property of the heat equation on Riemannian manifolds", Comm. Pure and Applied Math., 12(1959), 1-11.
%\bibitem[FF]{FF} Federer, H.; Fleming, W. H., Normal and integral currents. Ann. of Math. (2),
%72:458–520, 1960.
\bibitem[Gr]{GromovFilling} Gromov, M., Filling Riemannian Manifolds, J. Differential Geom., 18(1983), 1-147.
\par\noindent
%\bibitem[Gr]{Gro101} Gromov, M., 101 Questions, conjectures and problems around scalar curvature,
%https://www.ihes.fr/~gromov/wp-content/uploads/2018/08/101-problemsOct1-2017.pdf  .%
\bibitem[Gu10]{Gu11} Guth, L., ``Systolic inequalities and minimal hypersurfaces", Geom. Functional Analysis (GAFA), 19(2010), 1688-1692.

\bibitem[Gu11]{Gu1} Guth, L.,  Volumes of balls in large Riemannian manifolds. Ann. of Math. (2) 173 (2011), no. 1, 51-76. 
%\bibitem[Gu13]{Gu3} Guth, L., Contraction of areas vs. topology of mappings, Geom. Funct. Anal.,
\par\noindent
%23(6):1804-1902, 2013.
\bibitem[Gu17]{Guth_Urysohn} Guth, L.,   Volumes of balls in  Riemannian manifolds and Urysohn width, J. Top. Anal., 9(2)(2017), 195-219. 

\bibitem[K]{K} Katz, M., The filling radius of two-point homogeneous spaces, J. Differential Geometry, 18(1983), 505-510.
%\bibitem[LW]{LW} L.H. Loomis and H. Whitney. An inequality related to the isoperimetric inequality. Bull. Amer. Math. Soc., 55:961-962, 1949
%\bibitem[MS]{MS} Michael, J.; Simon, L., Sobolev and mean-value inequalities on generalized submanifolds of $R^n$, Comm. Pure Appl. Math 26 (1973), 361-379.
%\bibitem[W]{W} Wenger, S., A short proof of Gromov's filling inequality, Proc. Amer. Math. Soc., 136(8)(2008), 2937-2941.
%\bibitem[Y]{Y} Young, R., Quantitative nonorientability of embedded cycles, 
%Duke Math. J. 167(1) (2018), 41-108.
\bibitem[LLNR]{LLNR} Liokumovich, Y., Lishak, B., Nabutovsky, A., Rotman, R.,  ``Filling metric spaces", arXiv:1905.06522 v.1; version 3 will appear in Duke Math. J.
\par\noindent
\bibitem[N]{N} Nabutovsky, A., ``Linear bounds for constants in Gromov's systolic inequalities and related results", aXiv:1909.12225 v.1. 
\par\noindent
\bibitem[P]{P} Papasoglu, P., ``Uryson width and volume", GAFA 30(2020), 574-587.
\bibitem[S]{S} A.S. Schwarz, `` The genus of a fiber space", Amer. Math. Soc. Transl. 55 (1966), no. 2, 49–140.
\bibitem[W]{W} Wenger, S., A short proof of Gromov's filling inequality, Proc. Amer. Math. Soc., 136(8)(2008), 2937-2941.
\end{thebibliography}
\end{document}